\documentclass[12pt]{amsart}

\usepackage{epsf,amsmath,amssymb,latexsym,amsthm}

\newtheorem{theorem}{Theorem}[section]
\newtheorem{proposition}[theorem]{Proposition}

\newtheorem{lemma}[theorem]{Lemma}
\newtheorem{Obs}[theorem]{Observation}

\newcommand{\grn}{G_{r,n}}
\newcommand{\sumlim}{\sum\limits}

\begin{document}

\pagenumbering{arabic}
\pagestyle{headings}
\def\sof{\hfill\rule{2mm}{2mm}}
\def\ls{\leq}
\def\gs{\geq}
\def\SS{\mathcal S}
\def\qq{{\bold q}}
\def\txx{{\frac1{2\sqrt{x}}}}

\title[Colored descents in $\grn$]{Counting descent pairs with prescribed colors in the colored permutation groups}

\author{Eli Bagno}
\address{The Jerusalem College of Technology, Jerusalem, Israel}
\email{bagnoe@jct.ac.il}

\author{David Garber}
\address{Department of Applied Mathematics, Faculty of Sciences, Holon Institute of Technology, PO Box 305,
58102 Holon, Israel} \email{garber@hit.ac.il}

\author{Toufik Mansour}
\address{Department of Mathematics, University of Haifa, 31905 Haifa,
Israel.}
\email{toufik@math.haifa.ac.il}

\begin{abstract} We define new statistics,
{\it $(c,d)$-descents}, on the colored permutation groups
$\mathbb{Z}_r \wr S_n$ and compute the distribution of these
statistics on the elements in these groups. We use some
combinatorial approaches, recurrences, and generating functions
manipulations to obtain our results.
\end{abstract}

\date{\today}

\maketitle

\section{Introduction}

The {\it colored permutation group} is the wreath product
$G_{r,n}=\mathbb{Z}_r \wr S_n=\mathbb{Z}_r^n \rtimes S_n$,
consisting of all the pairs $(z,\tau)$, where $z$ is an $n$-tuple of
integers between $0$ and $r-1$ and $\tau \in S_n$. The
multiplication is defined by the following rule: For
$z=(z_1,\dots,z_n)$ and $z'=(z'_1,\dots,z'_n)$ \[(z,\tau) \cdot
(z',\tau')=((z_1+z'_{\tau^{-1}(1)},\dots,z_n+z'_{\tau^{-1}(n)}),\tau
\circ \tau')\] (here $+$ is taken modulo $r$). We usually write an
element in $G_{r,n}$ as $\pi = \pi_1^{[c_1]} \pi_2 ^{[c_2]} \cdots
\pi_n^{[c_n]}$ where $c_i=z_{\pi^{-1}(i)}$, i.e. $c_i$ is the color
of $\pi(i)$.

The colored permutation group $G_{r,n}$ can also be seen as the set
of all permutations of the set:
$$\Sigma_{r,n}=\{1,\dots,n,\bar{1},\dots,\bar{n},\dots,1^{[r-1]},\dots,n^{[r-1]}\}$$
satisfying $\pi(\bar{i})=\overline{\pi(i)}$. If $\pi_k=i^{[j]}$, we
define $|\pi_k|=i$.

 The classical
Weyl groups appear as special cases:
 the symmetric group $G_{1,n}=S_n$ and the
hyperoctahedral group $G_{2,n}=B_n$. In the last case, the
alphabet $B_n$ is acting on can be written as
$\Sigma=\{\pm1,\dots,\pm n\}$ or as
$\Sigma=\{1,\dots,n,\bar{1},\dots,\bar{n}\}$.

 On the symmetric group $S_n$, a {\it descent pair} of a permutation $\pi \in S_n$ is a pair
 $(i,i+1)$ such that $\pi_i>\pi_{i+1}$. The {\it descent set} of a permutation
 $\pi=\pi_1\pi_2\cdots\pi_n\in S_n$, denoted by ${\rm Des}(\pi)$, is the set
of indices $i$ such that $(i,i+1)$ is a descent pair. The {\it
number of descents} in a permutation $\pi$, denoted by ${\rm
des}(\pi)=|{\rm Des}(\pi)|$, is a classical permutation statistic.
This statistic was first studied by MacMahon \cite{Mac} almost a
century ago, and it still plays an important role in the study of
permutation statistics.

Kitaev and Remmel \cite{KR} counted descents in $S_n$ according to
the parity of the first or second element of the descent pair. In
\cite{KR0},  they generalize it to the case where the first or the
second element in the descent pair is divisible by $k$ for some $k
\geq 2$. Some more research was done in the direction of studying
the corresponding distribution for words \cite{HLR,HR,KMR}.

\medskip

In this paper, we introduce new statistics, {\it positive descent},
{\it negative descent} and {\it pn-descent} on $B_n$, which depend
on the signs of the first or the second element of the descent pair.

We get the following results:

\begin{proposition}\label{thm1}
The number of signed permutations in $B_n$ with exactly $m$
pn-descents is given by $n!\binom{n+1}{n-2m}$.
\end{proposition}

\begin{proposition}\label{num_pos_des}
The number of signed permutations in $B_n$ with exactly $m$ positive
(negative) descents is given by
$$\sum_{i\geq0}\sum_{j=0}^i\sum_{k=0}^j(-1)^{i+j+k+n+m}k^{n+j-i}(i-j)!\binom{n}{i-j}\binom{i}{j}\binom{j}{k}\binom{n-i}{m}.$$
\end{proposition}

\medskip

We generalize the studying of these statistics to the colored
permutation group $G_{r,n}$. We define for each two colors $c\leq d$
the {\it $(c,d)-$descent} and compute the number of elements of
$\grn$ having a fixed number of $(c,d)-$descents.

We get the following results:

\begin{proposition}\label{cddes}
The number of colored permutations in $G_{r,n}$ with exactly $m$
$(c,d)$-descents, $0\leq c<d\leq r-1$, is given by
$$n!\sum_{j=0}^{n-2m}\binom{m+j}{j}\binom{j}{n-2m-j}r^{2j+2m-n}(-1)^{n-j}.$$
\end{proposition}

\begin{proposition}\label{ccdes}
The number of colored permutations in $G_{r,n}$ with exactly $m$
$(c,c)$-descents, $0\leq c \leq r-1$, is given by
$$n!\sum\limits_{j=0}^{n-m}\sum\limits_{i=0}^j\sum\limits_{k=0}^i\binom{j}{i}\binom{i}{k}\binom{n-j}{m}\frac{(-1)^{n+m+j}(1-r)^{i-k}k^{n-j+i}}{(n-j+i)!}.$$
\end{proposition}

\medskip

The paper is organized as follows. In Section \ref{defs}, we define
positive and negative descents on $B_n$ and $(c,d)-$descents on
$G_{r,n}$. In Section \ref{enumerate_bn}, we enumerate the positive,
negative and pn-descents on $B_n$, so we prove Propositions
\ref{thm1} and \ref{num_pos_des}. In Section \ref{enumerate grn}, we
enumerate $(c,d)$-descents in $G_{r,n}$ for $c \leq d$, so we prove
Propositions \ref{cddes} and \ref{ccdes}.

\section{New descent statistics on $B_n$ and $\grn$}\label{defs}

Assume that the alphabet $\Sigma=\{1,\dots, n, \bar{1},\dots
\bar{n}\}$ of $B_n$ is ordered as
$$\bar{1}<\cdots<\bar{n}(<0)< 1<\cdots <n.$$  We say that a
signed permutation $\pi \in B_n$ has a {\it positive} (resp. {\it
negative}) {\it descent} at the index $i$ if $\pi_{i}>\pi_{i+1}>0$
(resp. $0>\pi_{i}>\pi_{i+1}$), and $\pi$ has a {\it pn-descent} if
$\pi_{i}>0>\pi_{i+1}$. The number of positive descents, negative
descents and pn-descents in $\pi$ are denoted by ${\rm
pdes}(\pi)$, ${\rm ndes}(\pi)$ and ${\rm pndes}(\pi)$,
respectively. For example, if $\pi=14\bar{3}75\bar{6}\bar{2}8$,
then ${\rm pdes}(\pi)=1$, ${\rm ndes}(\pi)=1$ and ${\rm
pndes}(\pi)=2$. In Section \ref{enumerate_bn}, we give explicit
formulas for the number of permutations in $B_n$ with exactly $m$
positive descents (resp. negative descents, pn-descents).

\medskip

Fix the following order on $\Sigma_{r,n}$:
$$1^{[r-1]} < 2^{[r-1]} < \cdots < n^{[r-1]} < \cdots < 1^{[1]}< \cdots < n^{[1]} < 1^{[0]} < \cdots < n^{[0]}.$$
We extend our new statistics to the group of colored permutations
$G_{r,n}$:  Let $\pi \in \grn$ and let $c\leq d$. A {\em
$(c,d)-$descent} is a descent in position $i$ such that $\pi(i)$ is
colored by $c$ and $\pi(i+1)$ is colored by $d$. For example, if
$\pi=6^{[1]}2^{[2]}4^{[0]}5^{[1]}3^{[2]}1^{[2]}$, then $\pi$ has two
$(1,2)-$descents:  $i=1$ and $i=4$, a $(0,1)-$descent: $i=3$ and a
$(2,2)-$descent: $i=5$ . The number of $(c,d)-$descents in $\pi$ is
denoted by ${\rm des}_{c,d}(\pi)$. In Section \ref{enumerate grn},
we study the generating function for the number of colored
permutations in $G_{r,n}$ having exactly $m$ $(c,d)-$descents.

\section{Positive, negative and pn-descents in $B_n$}\label{enumerate_bn}

In this section, we find a formula for the number of signed
permutations in $B_n$ with exactly $m$ positive descents (resp.
negative descents, pn-descents). Using the bijection
$\pi_1\pi_2\ldots\pi_n\mapsto(-\pi_1)(-\pi_2)\ldots(-\pi_n)$, we
have that the number of permutations in $B_n$ with exactly $m$
positive descents is equal to the number of permutations in $B_n$
with exactly $m$ negative descents. In the following two
subsections, we enumerate the signed permutations in $B_n$ according
to the number of pn-descents and positive descents.

\subsection{pn-descents}

In this section, we prove Proposition \ref{thm1}.

 Let $\pi =\pi_1 \pi_2 \cdots \pi_n \in
B_n$. A {\it block} in $\pi$ is a subsequence $\pi_i \pi_{i+1}
\cdots \pi_j$ of $\pi$ such that for all $k \in [i,j]$, $\pi_k$
has a common sign. A block is called {\it positive} (resp. {\it
negative}) if $\pi_i>0$ (resp. $\pi_i<0$). The cardinalities of
the blocks form a composition of $n$.  For example, if $\pi = 3 4
\bar{1} \bar{5} 2 7 6 \bar{8}$, then the corresponding blocks of
$\pi$ are $\{3,4\}, \{ \bar{1},\bar{5} \}, \{ 2,7,6 \}, \{
\bar{8}\}$ while the corresponding composition is: $(2,2,3,1)$.
 A simple observation is that the pn-descents appear exactly in the transitions
from a positive block to a negative one. Hence, in order to
enumerate the pn-descents, we have to enumerate these transitions.

We denote the number of blocks in $\pi$ by $b(\pi)$. Given a signed
permutation $\pi$, we have:
$${\rm pndes}(\pi)= \left\{
\begin{array}{lc}
\frac{b(\pi)}{2} & b(\pi)\equiv 0 ({\rm mod}\ 2), \mbox{ first block
is
positive}\\
\frac{b(\pi)}{2}-1 & b(\pi)\equiv 0 ({\rm mod}\ 2), \mbox{ first
block is
negative}\\
\frac{b(\pi)-1}{2} & b(\pi)\equiv 1 ({\rm mod}\ 2) \\
\end{array}\right.$$

Hence, in order to enumerate all the elements of $B_n$ having a
given number of pn-descents, we first have to go over all the
compositions of $n$ into non-trivial parts. Such a composition gives
the structure of the blocks. Note that we have two possibilities
according to the sign of the first block. Then we can fill in the
numbers in $n!$ ways. We denote by ${\rm Comp}(n)$ the set of
compositions of $n$ into non-trivial parts. Thus, we have:
\begin{eqnarray*}
\hspace{-35pt}\sum_{\pi\in B_n} q^{{\rm pndes}(\pi)} & = &
n! \sum_{\tiny \begin{array}{c} \lambda \in {\rm Comp}(n)\\ |\lambda|\ {\rm even}\end{array}} q ^{\frac{|\lambda|}{2} -1}+ n! \sum_{\tiny \begin{array}{c} \lambda \in {\rm Comp}(n)\\ |\lambda|\ {\rm even}\end{array}} q ^{\frac{|\lambda|}{2}} +\\
& & +2 n! \sum_{\tiny \begin{array}{c} \lambda \in {\rm Comp}(n)\\ |\lambda|\ {\rm odd}\end{array}} q ^{\frac{|\lambda|-1}{2}}=  \\
& = & n! \left( \sum_{k=1}^{\lfloor \frac{n}{2} \rfloor} \sum_{\tiny \begin{array}{c} \lambda \in {\rm Comp}(n)\\ |\lambda|=2k \end{array}} q ^k \left( 1+\frac{1}{q} \right)+  2 \sum_{k=0}^{\lfloor \frac{n}{2} \rfloor}\sum_{\tiny \begin{array}{c} \lambda \in {\rm Comp}(n)\\ |\lambda|=2k+1 \end{array}} q ^k \right) = \\
& = & n! \left( \sum_{k=1}^{\lfloor \frac{n}{2} \rfloor} \binom{n-1}{n-2k} q ^k \left( 1+\frac{1}{q} \right)+  2 \sum_{k=0}^{\lfloor \frac{n}{2} \rfloor} \binom{n-1}{n-2k-1} q ^k \right) \\
\end{eqnarray*}

By the Pascal identity
$\binom{a-1}{b}+\binom{a-1}{b-1}=\binom{a}{b}$, we have that the
coefficient of $q^m$, $m \in \{0,\dots,\lfloor{\frac{n}{2}}\rfloor
\}$, is given by
\begin{eqnarray*}
n!\left( \binom{n-1}{n-2m} + 2\binom{n-1}{n-2m-1}+
\binom{n-1}{n-2m-2} \right) & =& n! \binom{n+1}{n-2m},
\end{eqnarray*} as stated in Proposition \ref{thm1}.

\subsection{Positive
and negative descents}

In this section, we find a formula for the number of signed
permutations with exactly $m$ positive descents.

Let
$$p_n(q)=\sumlim_{\pi\in B_n}q^{{\rm pdes}(\pi)}$$ be
the generating function for the positive descents in $B_n$.

In order to find a recurrence for $p_n(q)$,  we define two more
generating functions:
$$p_n^+(q)=\sum\limits_{\pi\in B_n, \ \pi_1>0}q^{{\rm
pdes}(\pi)}$$
$$p_n^-(q)=\sum\limits_{\pi\in B_n, \ \pi_1<0}
q^{{\rm pdes}(\pi)}.$$ Note that $p_n(q)=p_n^+(q)+p_n^-(q)$.

\medskip

Let $\pi \in B_n$. If $|\pi_n|=n$, denote $\pi=\pi'\pi_n$ and we
have $${\rm pdes}(\pi)={\rm pdes}(\pi').$$

Otherwise, let $j<n$ be such that $|\pi_j|=n$. Denote: $\pi=\pi'
\pi_j \pi'' \in B_n$. If $\pi_j=\bar{n}$, then ${\rm
pdes}(\pi)={\rm pdes}(\pi')+{\rm pdes}(\pi'')$. If $\pi_j=n$, we
have two cases depending on the sign of $\pi''_1$ (the first
element of $\pi''$):
$${\rm pdes}(\pi)=\left\{ \begin{array}{lc} {\rm pdes}(\pi')+{\rm pdes}(\pi'') & \pi_1''<0 \\
1+{\rm pdes}(\pi')+{\rm pdes}(\pi'') & \mbox{otherwise}
\end{array}\right.$$
Since there are no restrictions on the positions of the digits\break
$1,2,\dots,n-1$, we have the following recurrences:

\begin{lemma}\label{lem1}
For all $n\geq1$,
\begin{eqnarray*}
p_n(q) &=& 2p_{n-1}(q)+\sum_{j=1}^{n-1}\binom{n-1}{j-1}p_{j-1}(q)p_{n-j}(q)+\\
& &+\sum_{j=1}^{n-1}\binom{n-1}{j-1}p_{j-1}(q)(p_{n-j}^-(q)+qp_{n-j}^+(q)),\\
p_n^+(q) &=& 2p_{n-1}^+(q)+\sum_{j=2}^{n-1}\binom{n-1}{j-1}p_{j-1}^+(q)p_{n-j}(q)+\\
& &
+\sum_{j=1}^{n-1}\binom{n-1}{j-1}p_{j-1}^+(q)(p_{n-j}^-(q)+qp_{n-j}^+(q)),
\end{eqnarray*}

\end{lemma}

\begin{proof}
In both expressions, the first summand corresponds to the case
$|\pi_n|=n$, the second corresponds to the case $\pi_j=\bar{n}$ for
some $j$, $1 \leq j \leq n-1$, and the third corresponds to the case
$\pi_j=n$ for some $j$, $1 \leq j \leq n-1$.
\end{proof}

 In order to find an explicit formula for $p_n(q)$, we rewrite
the above recurrences in terms of exponential generating functions.
Define:
$$P(x,q)=\sum_{n\geq0}p_n(q)\frac{x^n}{n!},\quad
P^+(x,q)=\sum_{n\geq0}p_n^+(q)\frac{x^n}{n!},$$
$$P^-(x,q)=\sum_{n\geq0}p_n^-(q)\frac{x^n}{n!}.$$

Since $p_n(q)=p_n^+(q)+p_n^-(q)$, we have that
$P(x,q)=P^+(x,q)+P^-(x,q)$. Moreover, for all $n \geq 1$, Lemma
\ref{lem1} gives:
\begin{eqnarray*}
\frac{d}{dx} \left(\frac{p_n(q)}{n!}x^n\right)&=&\frac{2p_{n-1}(q)}{(n-1)!}x^{n-1}+x^{n-1}\sum_{j=0}^{n-2}\frac{p_{j}(q)}{j!}\cdot\frac{p_{n-1-j}(q)}{(n-1-j)!}+\\
& & +x^{n-1}\sum_{j=0}^{n-2}\frac{p_{j}(q)}{j!}\cdot\frac{p_{n-1-j}^-(q)+qp_{n-1-j}^+(q)}{(n-1-j)!},\\
\frac{d}{dx} \left(\frac{p_n^+(q)}{n!}x^n\right)&=&
\frac{p_{n-1}^+(q)}{(n-1)!}x^{n-1}+x^{n-1}\sum_{j=1}^{n-2}\frac{p_{j}^+(q)}{j!}\cdot\frac{p_{n-1-j}(q)}{(n-1-j)!}+\\
& & +x^{n-1}\sum_{j=0}^{n-2}\frac{p_{j}^+(q)}{j!}\cdot\frac{p_{n-1-j}^-(q)+qp_{n-1-j}^+(q)}{(n-1-j)!}.\\
\end{eqnarray*}
Summing over all $n\geq1$ and using the initial conditions
$p_0(q)=p_0^+(q)=1$ and $p_0^-(q)=0$, we get that:
\begin{eqnarray*}
P(x,q)&=&P^+(x,q)+P^-(x,q),\\
\frac{d}{dx}\left(P(x,q)\right)&=&(1-q)P(x,q)+(P(x,q))^2+P(x,q)P^-(x,q)+\\
& & +qP(x,q)P^+(x,q),\\
\frac{d}{dx}\left(P^+(x,q)\right)&=&(1-q)P^+(x,q)-P(x,q)+P^+(x,q)P(x,q)+\\
& & +P^+(x,q)P^-(x,q)+q(P^+(x,q))^2.
\end{eqnarray*}
Using any mathematical programming package, such as Maple or
Mathematica, we obtain the following result.

\begin{theorem}\label{gen fun pp}
The generating functions $P(x,q)$, $P^+(x,q)$ and $P^-(x,q)$ are
given by
$$P(x,q)=\frac{1-q}{xq-x-q+e^{(q-1)x}},$$
$$P^+(x,q)=\frac{(1-x)(1-q)}{xq-x-q+e^{(q-1)x}},$$
$$P^-(x,q)=\frac{x(1-q)}{xq-x-q+e^{(q-1)x}},$$
respectively.
\end{theorem}
We now calculate:
\begin{eqnarray*}
P(x,q)&=&\frac{1-q}{xq-x-q+e^{(q-1)x}}=\\
&=&\frac{1-q}{(x-1)(q-1)-1 +e^{(q-1)x}}=\\
&=&\frac{1}{1-\left(x-\frac{1-e^{(q-1)x}}{q-1} \right)}=\\
&=& \sum\limits_{i\geq0}\left(x-\frac{1-e^{(q-1)x}}{q-1}\right)^i=\\
&=&\sum_{i\geq0}\sum_{j=0}^i\sum_{k=0}^j\sum_{\ell\geq0}\frac{(-1)^{j+k}k^\ell}{\ell!}\binom{i}{j}\binom{j}{k}x^{i-j+\ell}(q-1)^{\ell-j}.
\end{eqnarray*}
Thus, the number of signed permutations in $B_n$ with exactly $m$
positive descents is $n!$ times the coefficient of $x^n q^m$ in
$P(x,q)$. So we get that this number is:
$$\sum_{i\geq0}\sum_{j=0}^i\sum_{k=0}^j(-1)^{i+j+k+n+m}k^{n+j-i}(i-j)!\binom{n}{i-j}\binom{i}{j}\binom{j}{k}\binom{n-i}{m},$$
as stated in Proposition \ref{num_pos_des}.

\section{$(c,d)$-descents in $G_{r,n}$}\label{enumerate grn}
In this section, we find a formula for the number of colored
permutations in $G_{r,n}$ with exactly $m$ $(c,d)-$descents, where
$c\leq d$. We split our treatment into two cases: $(c,d)-$descents
where $c<d$ and $(c,c)-$descents.

\subsection{$(c,d)$-descents with $c<d$}

We give two approaches to this enumeration, which give different
presentations for the same number. The first is based on a counting
argument and is actually a generalization of the proof of
Proposition \ref{thm1}, while the other uses recurrences and
manipulations of generating functions, as in the proof of
Proposition \ref{num_pos_des}.

\subsubsection{Counting approach}

Let $\pi =\pi_1^{[c_1]} \pi_2^{[c_2]} \cdots \pi_n^{[c_n]} \in
\grn$. Similar to the case of $B_n$, define a $c-$block (resp.
$\bar{c}-$block) in $\pi$ to be a subsequence $\pi_i^{[c_i]}
\pi_{i+1}^{[c_{i+1}]} \cdots \pi_j^{[c_j]}$ of $\pi$ such that for
all $k \in [i,j]$, $c_k=c$  (resp. $c_k\neq c$). As in the case of
$B_n$, the cardinalities of the blocks form a composition of $n$.
For example, in $G_{6,6}$, with $c=2, d=3$, if $\pi = 2^{[2]}
4^{[2]} 5^{[3]} 1^{[4]} 6^{[2]} 3^{[1]}$, then the corresponding
blocks of $\pi$ are $\{2^{[2]},4^{[2]}\}, \{5^{[3]}, 1^{[4]}\},
\{6^{[2]} \}, \{ 3^{[1]}\}$, while the corresponding composition is
$(2,2,1,1)$.

Note that $(c,d)-$descents can appear in the transitions from a
$c$-block to a $\bar{c}-$block, but not necessarily in all of these
transitions. Hence, in order to enumerate the $(c,d)-$descents, we
have to enumerate these transitions in a smart way.

Recall that ${\rm Comp}(n)$ is the set of compositions of $n$. Let
${\rm Comp^{clr}}(n)={\rm Comp}(n) \times \{c,\bar{c}\}$. For each
$\varphi \in {\rm Comp}(n), x \in \{c, \bar{c}\}$, the element
$(\varphi,x) \in {\rm Comp^{clr}}(n)$ represents the composition
$\varphi$ where the first part is colored by $x$.

Now, let $\mu = (\varphi,x) \in {\rm Comp^{clr}}(n)$. We define the
following parameters:

\begin{itemize}
\item $e_\varphi$ is the sum of parts in the even places of $\varphi$.
For example, if $\varphi=(2,3,4,5)$ then $e_\varphi=3+5=8$.

\item $k=k_\mu=\left\{\begin{array}{cc} n-e_\varphi & x=c\\
                                         e_\varphi &
                                         x=\bar{c}.\end{array}\right.$
\item $b=b_\mu$ is the number of $c-$blocks and $\bar{b}_{\mu}$ is
the number of $\bar{c}-$blocks. Explicitly,
 $$b_\mu=\left\{\begin{array}{cc} \frac{|\varphi|}{2} & |\varphi|\ {\rm
 even}\\ & \\
                                         \frac{|\varphi|+1}{2} & |\varphi|\ {\rm odd},  x=c
                                         \\ & \\
                                         \frac{|\varphi|-1}{2} & |\varphi|\ {\rm odd},
                                         x=\bar{c},
                                         \end{array}\right.$$
where $|\varphi|$ is number of parts of $\varphi$, and
$\bar{b}_{\mu}=|\varphi|-b_{\mu}$.
\item $t_\mu$ is the number of transitions between a $c-$block
to a $\bar{c}-$block.

\end{itemize}

\medskip

Define:
$$A_{r,n}(q) = \sumlim_{\pi \in G_{r,n}} q^{{\rm des}_{c,d}(\pi)}$$
For calculating $A_{r,n}(q)$, instead of running over the elements
of $\grn$, we run over the elements of ${\rm Comp^{clr}}(n)$.

\begin{lemma}
For all $n\geq0$,
\begin{equation}\label{comp formula}
A_{r,n}(q)=n! \sumlim_{\mu \in {\rm Comp^{clr}}(n)}
(q+r-2)^{t_\mu}(r-1)^{n-k_\mu-t_\mu}.
\end{equation}
\end{lemma}
\begin{proof}
Each $\mu \in {\rm Comp^{clr}}(n)$ gives rise to $n!
(r-1)^{n-k_{\mu}}$ elements of $\grn$. $\pi\in \grn$ contributes a
$(c,d)-$descent in each transition from a $c-$block to a
$\bar{c}-$block such that the first digit of the $\bar{c}-$block is
colored by $d$.
\end{proof}

In order to get an explicit expression for Formula (\ref{comp
formula}), we treat separately ${\rm Comp}(n) \times \{c\}$ and
${\rm Comp}(n) \times \{\bar{c}\}$.

\medskip

Let $\mu=(\varphi,c) \in {\rm Comp}(n) \times \{c\}$.

We split the contribution into three cases according to $|\varphi|$.
\begin{enumerate}
\item $|\varphi|=1$:
 In this case we have $k_{\mu}=n$ and $t_{\mu}=0$, thus its
 contribution is:

$$n!\sumlim_{\tiny \begin{array}{c} \mu=(\varphi,c) \in {\rm Comp^{clr}}(n) \\|\varphi|=1 \end{array}}
(q+r-2)^{t_\mu}(r-1)^{n-k_\mu-t_\mu} = n!.$$

\item $|\varphi|$ even: In this case, $k_{\mu}$ can vary in
$\{1,\dots,n-1\}$ and hence  $b_{\mu} \in \{1,\dots ,k_{\mu}\}$.
Moreover, $\bar{b}_{\mu}=t_{\mu}=b_{\mu}$. Then we have that the
contribution is:
$$n!\sumlim_{\tiny \begin{array}{c} \mu=(\varphi,c) \in {\rm Comp^{clr}}(n) \\|\varphi| \mbox{ even} \end{array}}
(q+r-2)^{t_\mu}(r-1)^{n-k_\mu-t_\mu} = $$ $$=n! \sumlim_{k=1}^{n-1}
\sumlim_{b=1}^{k} {k-1 \choose k-b}{n-k-1 \choose n-k-b} (q+r-2)^b
(r-1)^{n-k-b}.$$

Note that the first binomial coefficient corresponds to the choice
of dividing the $k$ digits colored by $c$ into $b$ non-empty
blocks, while the second binomial coefficient corresponds to the
choice of dividing the remaining $n-k$ digits (which are not
colored by $c$) into $\bar{b}$ non-empty blocks.

\item $|\varphi|>1$ odd: In this case, $k_{\mu}$ can vary in
$\{2,\dots,n-1\}$ and hence $b_{\mu} \in \{2,\dots ,k_{\mu}\}$.
Moreover, $\bar{b}_{\mu}=b-1,t_{\mu}=b_{\mu}-1$. Then the
contribution is:
$$n!\sumlim_{\tiny \begin{array}{c} \mu=(\varphi,c) \in {\rm Comp^{clr}}(n) \\|\varphi| \mbox{ odd} \end{array}}
(q+r-2)^{t_\mu}(r-1)^{n-k_\mu-t_\mu} = $$ $$=n!
\sumlim_{k=2}^{n-1} \sumlim_{b=1}^{k} {k-1 \choose k-b}{n-k-1
\choose n-k-b+1} (q+r-2)^{b-1} (r-1)^{n-k-b+1}.$$
\end{enumerate}

\medskip

Now, let us treat the set ${\rm Comp}(n) \times \{\bar{c}\}$. Let
$\mu=(\varphi,\bar{c}) \in {\rm Comp}(n) \times \{\bar{c}\}$.

As in the case of ${\rm Comp}(n) \times \{c\}$, we split the
computation into three cases according to $|\varphi|$:
\begin{enumerate}
\item $|\varphi|=1$:
$$n!\sumlim_{\tiny \begin{array}{c} \mu=(\varphi,\bar{c}) \in {\rm Comp^{clr}}(n) \\|\varphi|=1 \end{array}}
(q+r-2)^{t_\mu}(r-1)^{n-k_\mu-t_\mu} = n!(r-1)^n.$$

\item $|\varphi|$ even: In this case, $k_{\mu}\in \{1,\dots,n-1\}$ and hence  $b_{\mu} \in \{1,\dots
,k_{\mu}\}$. Moreover, $\bar{b}_{\mu}=t_{\mu}=b_{\mu}-1$. Then we
have:
$$n!\sumlim_{\tiny \begin{array}{c} \mu=(\varphi,\bar{c}) \in {\rm Comp^{clr}}(n) \\|\varphi| \mbox{ even} \end{array}}
(q+r-2)^{t_\mu}(r-1)^{n-k_\mu-t_\mu} = $$ $$=n!
\sumlim_{k=1}^{n-1} \sumlim_{b=1}^{k} {k-1 \choose k-b}{n-k-1
\choose n-k-b} (q+r-2)^{b-1} (r-1)^{n-k-b+1}.$$

\item $|\varphi|>1$ odd: In this case, $k_{\mu}$ can vary in
$\{1,\dots,n-2\}$ and hence $b_{\mu} \in \{2,\dots ,k_{\mu}\}$.
Moreover, $\bar{b}_{\mu}=b_{\mu}+1,t_{\mu}=b_{\mu}$. Then we have:
$$n!\sumlim_{\tiny \begin{array}{c} \mu=(\varphi,{\bar{c}}) \in {\rm Comp^{clr}}(n) \\|\varphi| \mbox{ odd} \end{array}}
(q+r-2)^{t_\mu}(r-1)^{n-k_\mu-t_\mu} = $$ $$=n!
\sumlim_{k=1}^{n-2} \sumlim_{b=1}^{k} {k-1 \choose k-b}{n-k-1
\choose n-k-b-1} (q+r-2)^{b} (r-1)^{n-k-b}.$$

\end{enumerate}

Summing up all the ingredients, we get:

\begin{eqnarray*}
 A_{r,n}(q)& =&n!(1+(r-1)^n)+ \\
& & \quad +n!\sumlim_{k=1}^{n-1} \sumlim_{b=1}^{k} {k-1 \choose
k-b}{n-k \choose n-k-b} (q+r-2)^{b}
(r-1)^{n-k-b}+\\
& & \quad +n!\sumlim_{k=1}^{n-1} \sumlim_{b=1}^{k} {k-1 \choose
k-b}{n-k \choose n-k-b+1} (q+r-2)^{b-1} (r-1)^{n-k-b+1}= \\
&=& n!\left(\frac{1-(r-1)^{n+1}}{2-r}\right)+ \\
& & \quad +n!\sumlim_{k=1}^{n-1} \sumlim_{b=1}^{k} {k \choose
b}{n-k \choose b} (q+r-2)^{b} (r-1)^{n-k-b}.
\end{eqnarray*}

Hence, we get the following result:
\begin{proposition}
The number of colored permutations in $G_{r,n}$ with exactly $0$
$(c,d)$-descents, $0\leq c<d\leq r-1$, is:
$$n!\left(\frac{1-(r-1)^{n+1}}{2-r}\right)+n!\sumlim_{k=1}^{n-1} \sumlim_{b=1}^{k} {k \choose b}{n-k
\choose b} (r-2)^{b} (r-1)^{n-k-b}.$$

The number of colored permutations in $G_{r,n}$ with exactly $m>0$
$(c,d)$-descents, $0\leq c<d\leq r-1$, is given by:
$$n!\sumlim_{k=1}^{n-1} \sumlim_{b=1}^{k} {k \choose b}{n-k \choose
b} {b \choose m} (r-2)^{b-m} (r-1)^{n-k-b}.$$
\end{proposition}

The numbers given in the last proposition turn out to be the same as
in Proposition \ref{cddes}.

\subsubsection{Recursive approach}
In this section, we find an explicit formula for $A_{r,n}(q)$ by
using a recursive approach.

Recall that
\[A_{r,n}(q)=\sumlim_{\pi\in G_{r,n}}q^{{\rm
des}_{c,d}(\pi)},\] and define:

$$A_{r,n}(q;c_1)=\sum_{\pi \in G_{r,n}, \
c_1(\pi)=c_1}q^{{\rm des}_{c,d}(\pi)}$$ and

$$A_{r,n}(q;c_1,c_2)=\sum_{\pi \in G_{r,n}, \
c_1(\pi)=c_1,c_2(\pi)=c_2}q^{{\rm des}_{c,d}(\pi)}.$$

By the definitions, we have
\begin{equation}\label{eqcc1}
A_{r,n}(q)= \sum_{j=0}^{r-1}A_{r,n}(q;j).
\end{equation}

Define the following two maps:
\begin{itemize}
\item For $\pi \in \grn$, define $\pi' \in G_{r,n-1}$ by
$$\pi' (i)=\left\{
\begin{array}{ccc}
\pi (i+1)    & &   \pi(i+1) < \pi(1) \\
\pi (i+1)-1  & &   \pi(i+1) > \pi(1) \\
\end{array}\right.$$
The idea for defining $\pi'$ is as follows: Write $\pi$ in its
complete notation, i.e. as a matrix of two rows. The first row of
$\pi'$ is $(1,2,\dots,n-1)$, while the second row is obtained from
the second row of $\pi$ by ignoring the digit $\pi(1)$ and the other
digits are placed in an order preserving way with respect to the
second row of $\pi$.

\item For $\pi \in \grn$, define $\pi'' \in G_{r,n-2}$ by
$$\pi'' (i)=\left\{
\begin{array}{ccc}
\pi (i+2)    & &   \pi(i+2) < \min\{\pi(1),\pi(2)\} \\
\pi (i+2)-1  & &   \min\{\pi(1),\pi(2)\} <\pi(i+2) < \max\{\pi(1),\pi(2)\} \\
\pi (i+2)-2  & &   \pi(i+2) > \max\{\pi(1),\pi(2)\} \\
\end{array}\right.$$
$\pi''$ differs from $\pi'$ only in the fact that the first row of
$\pi''$ is $(1,2,\dots,n-2)$ and in the second row we ignore the
digits $\pi(1)$ and $\pi(2)$.
\end{itemize}

\medskip

Let $\pi\in \grn$ be such that $c_1(\pi)=j$ with $j\neq c$. Since
each $(c,d)-$descent starts by an element colored by $c$, we have
${\rm des}_{c,d}(\pi)={\rm des}_{c,d}(\pi')$. Therefore,
$A_{r,n}(q;j)=nA_{r,n-1}(q)$ for each $j\neq c$. Hence, Equation
\eqref{eqcc1} gives
\begin{equation}\label{eqcc2}
A_{r,n}(q)=\sum_{j=0}^{r-1}
A_{r,n}(q;j)=(r-1)nA_{r,n-1}(q)+A_{r,n}(q;c).
\end{equation}
Again, by the definitions, we have for all $1 \leq i \leq n$:
\begin{eqnarray}\label{eqcc3}
A_{r,n}(q;c)&=& \sum_{\tiny\begin{array}{c}s=0\\s\neq c,d
\end{array}}^{r-1}A_{r,n}(q;c,s)+A_{r,n}(q;c,c)+A_{r,n}(q;c,d).
\end{eqnarray}

Now, let $\pi \in \grn$ be such that $c_1(\pi)=c,\ c_2(\pi)=s$.

\medskip

Then we have:
$${\rm des}_{c,d}(\pi)=\left\{ \begin{array}{cc}
                                            {\rm des}_{c,d}(\pi'') & s
                                            \neq c,d\\
                                            {\rm des}_{c,d}(\pi') &
                                            s =c \\
                                            {\rm des}_{c,d}(\pi'')+1 &
                                            s=d.\\
\end{array}\right.$$

Thus from Equation \eqref{eqcc3}, we get:
\begin{eqnarray*}
A_{r,n}(q;c)&=&(r-2)n(n-1)A_{r,n-2}(q)+nA_{r,n-1}(q;c)+ \\
& &+qn(n-1)A_{r,n-2}(q)=\\
&=&(q+(r-2))n(n-1)A_{r,n-2}(q)+nA_{r,n-1}(q;c).
\end{eqnarray*}

By substituting $A_{r,n}(q;c)$ from Equation \eqref{eqcc2} twice,
we obtain the following recurrence:
    $$A_{r,n}(q)-(r-1)nA_{r,n-1}(q)=nA_{r,n-1}(q)+(q-1)n(n-1)A_{r,n-2}(q)$$
which is equivalent to:
    $$A_{r,n}(q)=rnA_{r,n-1}(q)+(q-1)n(n-1)A_{r,n-2}(q),$$
for all $n\geq1$.

\medskip

In order to get an explicit formula for the number of colored
permutations in $G_{r,n}$ with exactly $m$ $(c,d)$-descents, we
rewrite the above recurrence relation in terms of generating
functions. Define:
$$A_r(x,q)=\sum_{n\geq0} A_{r,n}(q)\frac{x^n}{n!}.$$
By the above recurrence, and using the initial condition
$A_{r,0}(q)=1$, we get that:
    $$A_r(x,q)=1+rxA_r(x,q)+(q-1)x^2A_r(x,q),$$
which is equivalent to:
\begin{eqnarray*}
A_r(x,q)&=&\frac{1}{1-rx-(q-1)x^2}=
\frac{\frac{1}{1-rx+x^2}}{1-\left( \frac{qx^2}{1-rx+x^2} \right)}=\\
&=&\sum_{t\geq0}\frac{x^{2t}}{(1-rx+x^2)^{t+1}}q^t =
\sum_{t\geq0}\frac{x^{2t}}{(1-(x(r-x)))^{t+1}}q^t,
\end{eqnarray*}
which implies:
\begin{eqnarray*}
A_r(x,q)& =&
\sum_{t\geq0}\sum_{j\geq0}\binom{t+j}{j}(r-x)^jx^{2t+j}q^t=\\
&=&\sum_{t\geq0}\sum_{j\geq0}\sum_{i=0}^j\binom{t+j}{j}\binom{j}{i}
r^{j-i} (-1)^i x^{2t+j+i} q^t.
\end{eqnarray*}
Hence, the coefficient of $x^nq^m$ in $A_r(x,q)$ is given by:
$$\sum_{j=0}^{n-2m}\binom{m+j}{j}\binom{j}{n-2m-j}r^{2j+2m-n}(-1)^{n-j},$$
as stated in Proposition \ref{cddes}.

\medskip

Substituting $r=2$ in Proposition \ref{cddes}, and comparing with
Proposition \ref{thm1}, we obtain the following identity:
$$\sum_{j=0}^{n-2m}\binom{m+j}{j}\binom{j}{n-2m-j}2^{2j+2m-n}(-1)^{n-j}=\binom{n+1}{n-2m}.$$

\subsection{$(c,c)$-descents}
Let $\pi=i_1^{[j_1]}\cdots i_n^{[ j_n]} \in \grn$. Using a map $f:
\grn \rightarrow \grn$ which takes each colored digit $i^{[j]}$ to
$i^{[(j+1) \pmod r]}$ we obtain the following :

\begin{Obs}\label{gl2}
For each color $c \in \{0,\dots,r-1\}$, the number of colored
permutations with exactly $m$ $(c,c)$-descents is equal to the
number of colored permutations with exactly $m$ $(0,0)$-descents.
\end{Obs}

By the above observation, it suffices to find the generating
function for the number of colored permutations in $G_{r,n}$ having
$(0,0)$-descents.

\medskip

Let
$$g_{r,n}(q)= \sum_{\pi\in G_{r,n}}q^{{\rm des}_{0,0}(\pi)}.$$
In order to find a recurrence  for $g_{r,n}(q)$, we will use the
following notations: Define
$$g_{r,n}^+(q)=\sum\limits_{\pi\in G_{r,n}, \ c_1(\pi)=0}q^{{\rm des}_{0,0}(\pi)},$$
$$ g_{r,n}^-(q)=\sum\limits_{\pi\in G_{r,n}, \ c_1(\pi) \neq 0}q^{{\rm des}_{0,0}(\pi)}.$$
Using similar arguments as in the proof of Lemma \ref{lem1}, we get
the following lemma.

\begin{lemma}\label{gl3}
For all $n\geq1$,
\begin{eqnarray*}
g_{r,n}(q)&=&rg_{r,n-1}(q)+(r-1)\sum\limits_{j=1}^{n-1}\binom{n-1}{j-1}g_{r,j-1}(q)g_{r,n-j}(q)+\\
& & +\sum\limits_{j=1}^{n-1}\binom{n-1}{j-1}g_{r,j-1}(q)(g_{r,n-j}^-(q)+qg_{r,n-j}^+(q)),\\
g_{r,n}^+(q)&=&rg_{r,n-1}^+(q)+(r-1)\sum\limits_{j=2}^{n-1}\binom{n-1}{j-1}g_{r,j-1}^+(q)g_{r,n-j}(q)+\\
& &
+\sum\limits_{j=1}^{n-1}\binom{n-1}{j-1}g_{r,j-1}^+(q)(g_{r,n-j}^-(q)+qg_{r,n-j}^+(q)),\\
g_{r,n}(q)&=&g_{r,n}^+(q)+g_{r,n}^-(q).
\end{eqnarray*}
\end{lemma}

In order to find an explicit formula for $g_{r,n}(q)$, we rewrite
the above recurrences in terms of exponential generating functions.
Define
$$G_r(x,q)=\sum_{n\geq0}g_{r,n}(q)\frac{x^n}{n!},\quad
G^+_r(x,q)=\sum_{n\geq0}g_{r,n}^+(q)\frac{x^n}{n!},$$
$$G^-_r(x,q)=\sum_{n\geq0}g_{r,n}^-(q)\frac{x^n}{n!}.$$

By similar manipulations to the ones presented in the discussion
preceding Lemma \ref{gen fun pp}, we obtain:

\begin{theorem}
The generating functions $G_r(x,q)$, $G_r^+(x,q)$ and $G_r^-(x,q)$
are given by
$$G_r(x,q)=\frac{1-q}{(1-x)(1-q)-(r-1)+e^{(q-1)x}},$$
$$G_r^+(x,q)=\frac{(1-x)(1-q)}{(1-x)(1-q)-(r-1)+e^{(q-1)x}},$$
$$G_r^-(x,q)=\frac{x(1-q)}{(1-x)(1-q)-(r-1)+e^{(q-1)x}},$$
respectively.
\end{theorem}

In order to obtain an explicit formula for the number of colored
permutations in $G_{r,n}$ with exactly $m$ $(c,c)$-descents, we find
the coefficient of $x^nq^m$ in $G_r(x,q)$:
$$\begin{array}{ll}
G_r(x,q)&=\frac{1}{1-\left(x+\frac{e^{(q-1)x}+1-r}{q-1}\right)}\\
&=\sum\limits_{j\geq0}\left(x+\frac{e^{(q-1)x}+1-r}{q-1}\right)^j\\
&=\sum\limits_{j\geq0}\sum\limits_{i=0}^j\sum\limits_{k=0}^i\binom{j}{i}\binom{i}{k}\frac{x^{j-i}}{(q-1)^i}e^{(q-1)xk}(1-r)^{i-k}\\
&=\sum\limits_{j\geq0}\sum\limits_{i=0}^j\sum\limits_{k=0}^i\sumlim_{\ell\geq0}\binom{j}{i}\binom{i}{k}
\frac{x^{j-i+\ell}}{(q-1)^i\ell!}(q-1)^\ell k^\ell(1-r)^{i-k}.
\end{array}$$ Thus the coefficient of $x^nq^m$ in $G_r(x,q)$ is
given by
$$\sum\limits_{j=0}^{n-m}\sum\limits_{i=0}^j\sum\limits_{k=0}^i\binom{j}{i}\binom{i}{k}\binom{n-j}{m}\frac{(-1)^{n+m+j}(1-r)^{i-k}k^{n-j+i}}{(n-j+i)!}.$$
This implies that the number of colored permutations in $G_{r,n}$
with exactly $m$ $(c,c)$-descents is
$$n!\sum\limits_{j=0}^{n-m}\sum\limits_{i=0}^j\sum\limits_{k=0}^i\binom{j}{i}\binom{i}{k}\binom{n-j}{m}\frac{(-1)^{n+m+j}(1-r)^{i-k}k^{n-j+i}}{(n-j+i)!},$$
as stated in Proposition \ref{ccdes}.


\section*{Acknowledgements}
We would like to thank Ron Adin for some helpful advices.


\end{document}